\documentclass[11pt]{amsart}

\usepackage[T1]{fontenc}
\usepackage{amsmath,amssymb,amsthm,mathtools}
\usepackage[hidelinks]{hyperref}

\allowdisplaybreaks
\numberwithin{equation}{section}

\newif\ifanonymous
\ifdefined\ANONYMOUS
  \anonymoustrue
\else
  \anonymousfalse
\fi

\newtheorem{theorem}{Theorem}[section]
\newtheorem{lemma}[theorem]{Lemma}
\newtheorem{proposition}[theorem]{Proposition}
\newtheorem{corollary}[theorem]{Corollary}
\theoremstyle{definition}

\newtheorem{example}[theorem]{Example}
\theoremstyle{remark}
\newtheorem{remark}[theorem]{Remark}

\newcommand{\FS}{\operatorname{FS}}
\newcommand{\Z}{\mathbb Z}
\newcommand{\R}{\mathbb R}
\newcommand{\Q}{\mathbb Q}

\newcommand{\cS}{\mathcal S}
\newcommand{\cE}{\mathcal E}
\newcommand{\PiSet}{\mathcal P}
\newcommand{\sigmaf}{\sigma}

\title[Fixed-defect inverse theorems]
{Fixed-Defect Inverse Theorems for Subset Sums}

\ifanonymous
\else
  \author{Lizhong Chen}
  \address{Department of Mathematics, The Hong Kong University of Science and
  Technology, Clear Water Bay, Kowloon, Hong Kong}
  \email{lchendh@connect.ust.hk}
  \thanks{ORCID: 0009-0000-3495-7319}
\fi

\date{}

\ifanonymous
  \hypersetup{
    pdftitle={Fixed-Defect Inverse Theorems for Subset Sums},
    pdfauthor={},
    pdfsubject={Inverse theorems for sets with few subset sums},
    pdfkeywords={subset sums, inverse additive problems, stability,
      integer partitions, endpoint structure}
  }
\else
  \hypersetup{
    pdftitle={Fixed-Defect Inverse Theorems for Subset Sums},
    pdfauthor={Lizhong Chen},
    pdfsubject={Inverse theorems for sets with few subset sums},
    pdfkeywords={subset sums, inverse additive problems, stability,
      integer partitions, endpoint structure}
  }
\fi

\begin{document}

\begin{abstract}
Let \(A\) be an \(n\)-element set of positive real numbers, let
\(\FS(A)\) be its set of subset sums, and put
\(T_n=\binom{n+1}{2}\).  For every fixed integer \(C\geq-1\) and all
sufficiently large \(n\), we classify the sets satisfying
\[
  |\FS(A)|\leq T_n+n+C+1.
\]
Each such set is commensurable.  Its unique primitive integer
normalisation \(B\) either satisfies \(\sum B\leq T_n+n+C\) or belongs
to an explicit exceptional family specified by a missing element
\(m\in\{1,2\}\) and an integer partition of \(C+m\) or \(C+m+1\).
If \(P\) denotes the partition function, the exceptional family has exactly
\[
  P(C+1)+2P(C+2)+P(C+3)
\]
primitive dilation classes.  We also prove a local inverse theorem for
bounded increment excess.  If, for sufficiently large \(i\), adjoining
the largest element to the preceding \(i-1\) elements creates only
\(i+e\) new subset sums, where \(e\) is bounded, then the \(i\)-element
set is a dilation of \([1,i+e]_{\Z}\) with exactly \(e\) elements
deleted.  Conversely, every such deletion pattern has increment excess
\(e\).  The proof combines a
stabiliser argument in \(\R/x\Z\), Kneser's theorem, a quadratic
subset-sum bound, and endpoint propagation.  These arguments also give
effective commensurability and a finite-state encoding.  Together with
earlier results for \(C\leq-2\), this completes the eventual fixed-defect
classification for
every integer \(C\).
\end{abstract}

\maketitle

\section{Introduction and main results}
\label{sec:introduction}

Let \(A\) be a finite set of positive real numbers.  Its set of subset
sums is
\[
  \FS(A)=
  \left\{\sum_{a\in X}a:X\subseteq A\right\},
\]
where the empty sum is included.  We write
\(\sigmaf(A)=\sum_{a\in A}a\) and
\[
  T_n=\binom{n+1}{2}.
\]
If \(|A|=n\), the standard chain argument gives
\begin{equation}
\label{eq:intro-minimum}
  |\FS(A)|\geq T_n+1.
\end{equation}
Nathanson proved the corresponding integer result
\cite[Theorem~3]{Nathanson1995} and asked for a stability extension for
small \(h\)-subset-sum sets \cite[Section~6]{Nathanson1995}.  His
notation omits the empty sum; for a positive set, adjoining \(0\) increases
the support cardinality by one.  The proof of \eqref{eq:intro-minimum}
uses only positivity and order, so the same bound holds for positive real
sets.

Positive dilations of integer sets with small total sums provide one source
of sets with few subset sums.  For \(B\subset\Z_{>0}\),
\[
  \FS(B)\subseteq[0,\sigmaf(B)]_{\Z},
\]
so
\begin{equation}
\label{eq:intro-total-bound}
  |\FS(B)|\leq\sigmaf(B)+1.
\end{equation}
The total-sum bound does not account for every set with few subset sums.
A set can have a slightly larger total sum while its subset-sum support
omits only boundedly many integers near its two endpoints.  We determine
all such exceptions when the excess above the quadratic minimum is
\(n+O(1)\).

Carpenter, Defant and Kravitz established the first exact stability range.
For integers \(n\geq4\) and \(0\leq M\leq n-4\), they proved that an
\(n\)-element set \(A\subset\R_{>0}\) satisfies
\[
  |\FS(A)|\leq T_n+1+M
\]
if and only if \(A=\lambda B\) for some \(\lambda>0\) and
\(B\subset\Z_{>0}\) with
\[
  \sigmaf(B)\leq T_n+M
\]
\cite[Theorem~1.1]{CarpenterDefantKravitz2026}.  They also obtained a
quadratic-scale structural theorem that implies eventual
commensurability in the range considered here
\cite[Theorem~4.8]{CarpenterDefantKravitz2026}.

Write
\[
  M=n+C,
\]
where \(C\) is fixed and \(n\) tends to infinity.  The theorem of
Carpenter, Defant and Kravitz covers \(C\leq-4\) for all sufficiently
large \(n\).  Chen determined the two subsequent cases \(C=-3\) and \(C=-2\)
for every \(n\geq5\), including the sharp threshold and all equality
cases \cite[Theorems~4.1 and~5.9]{Chen2026Boundaries}.  That paper also
proved the endpoint self-loop and finite-to-infinite bridge used below
\cite[Lemma~3.3 and Theorem~3.4]{Chen2026Boundaries}.  The present paper
treats every fixed \(C\geq-1\).  Its classification holds for all
sufficiently large \(n\), and integer partitions uniquely parametrise
the exceptional sets.

For an integer \(w\), let
\(\PiSet(w)\) be the set of weakly increasing positive integer partitions
\[
  \pi=(\pi_1,\ldots,\pi_\ell),\qquad
  1\leq\pi_1\leq\cdots\leq\pi_\ell,\qquad
  |\pi|:=\sum_{r=1}^{\ell}\pi_r=w.
\]
Set
\[
  \PiSet(0)=\{()\},\qquad
  \PiSet(w)=\varnothing\quad(w<0),
\]
and write \(\ell(\pi)\) for the number of parts.  For
\(m\in\{1,2\}\) and \(\pi\in\PiSet(w)\), define
\begin{equation}
\label{eq:intro-template}
\begin{split}
  \cE_n(m,\pi)
  ={}&
  \left([1,n-\ell(\pi)+1]_{\Z}\setminus\{m\}\right)\\
  &{}\cup
  \left\{
  n-\ell(\pi)+r+1+\pi_r:
  1\leq r\leq\ell(\pi)
  \right\}.
\end{split}
\end{equation}
For \(n\) sufficiently large relative to \(w\), this is an
\(n\)-element set.  Its initial interval has one missing element, and the
parts of \(\pi\) give the upward displacements of its final elements.

A finite positive-real set is \emph{commensurable} if it is contained in a
positive dilation of \(\Z_{>0}\).  A finite set
\(B\subset\Z_{>0}\) is \emph{primitive} if \(\gcd(B)=1\).  Every
commensurable set has a unique representation \(A=\lambda B\), where
\(\lambda>0\) and \(B\) is primitive.  We call \(B\) the
\emph{primitive normalisation} of \(A\).

\begin{theorem}
\label{thm:main-classification}
Fix an integer \(C\geq-1\).  There is an effectively computable integer
\(N_C\) such that the following holds for every \(n\geq N_C\).  An
\(n\)-element set \(A\subset\R_{>0}\) satisfies
\begin{equation}
\label{eq:target-bound}
  |\FS(A)|\leq T_n+n+C+1
\end{equation}
if and only if exactly one of the following alternatives occurs.
\begin{enumerate}
\item There are unique \(\lambda>0\) and a primitive
  \(B\subset\Z_{>0}\) such that
  \[
    A=\lambda B,\qquad
    \sigmaf(B)\leq T_n+n+C.
  \]
\item There are unique
  \[
    \lambda>0,\qquad
    m\in\{1,2\},\qquad
    \pi\in\PiSet(C+m)\cup\PiSet(C+m+1)
  \]
  such that
  \[
    A=\lambda\cE_n(m,\pi).
  \]
\end{enumerate}
In the second alternative, if \(w=|\pi|\) and
\[
  L=\sigmaf(\cE_n(m,\pi))=T_{n+1}-m+w,
\]
then
\begin{equation}
\label{eq:intro-template-support}
  \FS(\cE_n(m,\pi))
  =[0,L]_{\Z}\setminus\{m,L-m\}.
\end{equation}
\end{theorem}

The first alternative is the \emph{ordinary total branch}; its sufficiency
follows immediately from \eqref{eq:intro-total-bound}.  We call the sets
in the second alternative \emph{exceptional}.  Formula
\eqref{eq:intro-template-support} gives
\[
  |\FS(\cE_n(m,\pi))|=L-1.
\]
If \(w=C+m\), then
\[
  L=T_n+n+C+1,\qquad
  |\FS(\cE_n(m,\pi))|=T_n+n+C.
\]
If \(w=C+m+1\), then
\[
  L=T_n+n+C+2,\qquad
  |\FS(\cE_n(m,\pi))|=T_n+n+C+1.
\]
In both cases the total exceeds \(T_n+n+C\), so the two alternatives in
Theorem~\ref{thm:main-classification} are disjoint.

Let \(P(w)=|\PiSet(w)|\), with \(P(0)=1\) and \(P(w)=0\) for
\(w<0\).

\begin{corollary}
\label{cor:template-count}
For each fixed integer \(C\geq-1\) and every sufficiently large \(n\),
the number of exceptional primitive dilation classes satisfying
\eqref{eq:target-bound} is
\[
  P(C+1)+2P(C+2)+P(C+3).
\]
\end{corollary}

The permitted partition weights contribute \(P(C+1)+P(C+2)\) when
\(m=1\), and \(P(C+2)+P(C+3)\) when \(m=2\).  Thus there are five
exceptional classes for \(C=-1\) and eight for \(C=0\); for fixed
\(C\), the number is independent of \(n\).

Combining Theorem~\ref{thm:main-classification} with
\cite[Theorem~1.1]{CarpenterDefantKravitz2026} and
\cite[Theorems~4.1 and~5.9]{Chen2026Boundaries} completes the eventual
fixed-defect classification for every integer \(C\).  For \(C=-3\)
and \(C=-2\), Chen's theorems retain the sharper statements for every
\(n\geq5\), including all equality cases.  The present theorem covers
\(C\geq-1\).

The main proof uses a local inverse theorem.  Let
\[
  A_i=\{a_1<\cdots<a_i\}\subset\R_{>0},
\]
let \(y_i\) be the increment excess defined in
Section~\ref{sec:preliminaries}.  Theorem~\ref{thm:bounded-excess}
states that, for fixed \(E\) and sufficiently large \(i\),
\[
  0\leq y_i=e\leq E
\]
if and only if
\[
  A_i=t\bigl([1,i+e]_{\Z}\setminus D\bigr),
  \qquad t>0,\quad |D|=e.
\]
No integrality or global support bound is assumed.  Bounded excess at
one step forces the entire prefix into a one-dimensional lattice and
determines the number of deleted elements.

Related literature considers subset sums in groups and vector spaces.
Brickell and Saks developed rank-sensitive lower bounds for
subset sums of finite vector sets \cite{BrickellSaks1993}.  DeVos, Goddyn,
Mohar and {\v S}amal proved a quadratic lower bound in an arbitrary
abelian group in terms of the stabiliser of the subset-sum set
\cite{DeVosEtAl2007}.  Balandraud, Girard, Griffiths and Hamidoune obtained
related estimates for symmetric and antisymmetric sets in finite abelian
groups \cite{BalandraudEtAl2013}.  Our local theorem uses the stabiliser
estimate in the circle quotient determined by the largest element.  The
order on the original positive-real set then converts the finite subgroup
obtained in that quotient into a deleted initial interval.

To see the group-theoretic mechanism, put \(x=a_i\) and reduce the earlier
elements modulo \(x\) in \(\R/x\Z\).  If \(R\) is the resulting residue
set and \(\rho=|\FS(R)|\), the chain decomposition gives
\[
  i\leq\rho\leq i+E.
\]
The quadratic estimate of DeVos, Goddyn, Mohar and {\v S}amal rules out a
trivial stabiliser.  Kneser's theorem then forces \(\FS(R)\) to be a finite
subgroup of \(\R/x\Z\).  The structure of this subgroup yields an
initial interval with finitely many deletions.  Lev's consecutive
subset-sum theorem confines all holes to two endpoint intervals; the
endpoint self-loop then identifies their number with the increment excess.

The local theorem also yields effective commensurability.  The
identity
\[
  |\FS(A)|=1+T_n+\sum_{i=1}^n y_i
\]
gives a linear upper bound on the nonnegative excesses.  If the rational
rank first increases at a late index, the preceding support and its
translate lie in distinct cosets, making the increment quadratic.
The first rank increase therefore occurs at an index bounded in terms
of \(C\).  The linear bound then forces a unit-excess step at a
sufficiently large index.  By Theorem~\ref{thm:bounded-excess}, the
corresponding prefix is commensurable, a contradiction.  This proves
effective commensurability without using a general vector-rank theorem.

After primitive normalisation, the same bound yields a unit-excess
prefix of the form
\[
  [1,i+1]_{\Z}\setminus\{m\}.
\]
If \(m\geq3\), its support is an interval and remains so under the
bounded-excess tail, placing the final set in the ordinary branch.  Thus an
exceptional set has \(m=1\) or \(m=2\).  Endpoint propagation gives
\[
  b_j=j+y_j
\]
on the remaining tail.  Since the \(b_j\) are strictly increasing, the
\(y_j\) are nondecreasing.  After omitting the initial run of indices for
which \(y_j=1\), the positive values \(y_j-1\) form the partition \(\pi\) in
\eqref{eq:intro-template}.  The support bound leaves precisely the two
weights \(C+m\) and \(C+m+1\).

The partition description also gives a finite-state encoding: for fixed
\(C\), a state records \(m\), the most recent partition part, and the
accumulated weight.  Section~\ref{sec:consequences} establishes the
corresponding bijection.

Section~\ref{sec:preliminaries} sets up the increment framework and the
additive tools.  Sections~\ref{sec:bounded-excess}--\ref{sec:endpoint-tail}
prove the local inverse theorem, effective commensurability, and endpoint
rigidity.  Sections~\ref{sec:classification} and~\ref{sec:consequences}
give the classification and its enumerative consequences.

\section{Increment growth and additive tools}
\label{sec:preliminaries}

\subsection{The increment framework}

For integers \(u\leq v\), write
\[
  [u,v]_{\Z}=\{u,u+1,\ldots,v\},
\]
and set \([u,v]_{\Z}=\varnothing\) when \(u>v\).  We abbreviate
\([1,v]_{\Z}\) to \([v]\).  As above, \(\FS(A)\) includes the empty sum.

Let
\[
  A=\{a_1<\cdots<a_n\}\subset\R_{>0}.
\]
For \(0\leq i\leq n\), put
\[
  A_i=\{a_1,\ldots,a_i\},\qquad
  S_i=\FS(A_i),
\]
where \(A_0=\varnothing\) and \(S_0=\{0\}\).  Define
\[
  d_i=|S_i\setminus S_{i-1}|,\qquad
  y_i=d_i-i,\qquad
  Y_i=\sum_{j=1}^i y_j.
\]
We refer to \(y_i\) as the \emph{increment excess} at \(i\).

\begin{lemma}
\label{lem:chain-count}
Let \(B=\{b_1<\cdots<b_m\}\subset\R_{>0}\), let \(x>b_m\), and put
\(S=\FS(B)\).  Decompose \(S\) into maximal arithmetic chains of common
difference \(x\).  Then
\[
  |(x+S)\setminus S|
\]
is the number of these chains and is at least \(m+1\).
\end{lemma}

\begin{proof}
Addition of \(x\) sends each nonterminal member of a chain to its next
member, but sends the upper endpoint outside \(S\).  Thus each chain
contributes exactly one point to \((x+S)\setminus S\).  There are at least
\(m+1\) chains because the points
\[
  0,b_1,\ldots,b_m
\]
lie in distinct chains: the absolute difference between any two is
positive and smaller than \(x\).
\end{proof}

The lemma gives \(y_i\geq0\) and
\begin{equation}
\label{eq:prefix-cardinality}
  |S_i|=1+T_i+Y_i,\qquad
  T_i=\binom{i+1}{2}.
\end{equation}
In particular,
\begin{equation}
\label{eq:global-budget}
  |\FS(A)|\leq T_n+n+C+1
  \quad\Longleftrightarrow\quad
  Y_n\leq n+C.
\end{equation}

The equality case in Lemma~\ref{lem:chain-count} is due to Carpenter,
Defant and Kravitz.

\begin{theorem}[Carpenter--Defant--Kravitz {\cite[Lemma~2.1]{CarpenterDefantKravitz2026}}]
\label{thm:cdk-equality}
Let \(m\geq3\), let \(B\subset\R_{>0}\) have \(m\) elements, and let
\(x>\max B\).  Then
\[
  |(x+\FS(B))\setminus\FS(B)|=m+1
\]
if and only if
\[
  B=t[m],\qquad x=t(m+1)
\]
for some \(t>0\).
\end{theorem}

Consequently, for \(i\geq4\),
\begin{equation}
\label{eq:zero-excess}
  y_i=0
  \quad\Longleftrightarrow\quad
  A_i=t[i]\ \text{for some }t>0.
\end{equation}

A finite set of positive real numbers is commensurable precisely when its
\(\Q\)-linear span has dimension one.  In that case, clearing denominators
and dividing by a greatest common divisor gives the primitive normalisation
described in the introduction.

\subsection{Subset sums in abelian groups}

For a subset \(X\) of an abelian group \(G\), its stabiliser is
\[
  \operatorname{stab}(X)=\{g\in G:g+X=X\}.
\]
We use the following standard form of Kneser's theorem, recorded in
\cite[Theorem~1.3]{DeVosEtAl2007}.

\begin{theorem}[Kneser]
\label{thm:kneser}
Let \(X_1,\ldots,X_s\) be finite nonempty subsets of an abelian group, and
let
\[
  H=\operatorname{stab}(X_1+\cdots+X_s).
\]
Then
\[
  |X_1+\cdots+X_s|
  \geq\sum_{\nu=1}^s|X_\nu+H|-(s-1)|H|.
\]
\end{theorem}

The second group-theoretic input is a quadratic estimate for subset sums.

\begin{theorem}[DeVos--Goddyn--Mohar--{\v S}amal
{\cite[Theorem~1.5]{DeVosEtAl2007}}]
\label{thm:devos}
Let \(R\) be a finite subset of an abelian group, and let
\[
  H=\operatorname{stab}(\FS(R)).
\]
Then
\[
  |\FS(R)|\geq |H|+\frac{1}{64}|R\setminus H|^2.
\]
\end{theorem}

Both theorems apply to infinite abelian groups.  We shall use them in the
circle group \(\R/x\Z\).

\subsection{Endpoint states}

For an integer \(L\geq0\) and a finite set
\[
  H\subseteq[1,\lfloor L/2\rfloor]_{\Z},
\]
define
\[
  \cS(L,H)
  =[0,L]_{\Z}\setminus\bigl(H\cup(L-H)\bigr),
\]
where \(L-H=\{L-h:h\in H\}\).  The set \(H\) records the lower endpoint
holes.  Subset-sum supports of positive-integer sets are symmetric about
half their total, so a support containing a central interval has this
form.

We use Lev's central interval theorem.

\begin{theorem}[Lev {\cite[Theorem~1]{Lev1998}}]
\label{thm:lev}
Let \(B\subseteq[1,\ell]_{\Z}\) have \(s\) elements.  If
\[
  \ell\leq\frac{3s}{2}-2,
\]
then
\[
  [2\ell-2s+1,\,
    \sigmaf(B)-(2\ell-2s+1)]_{\Z}
  \subseteq\FS(B).
\]
\end{theorem}

We also use the following endpoint results from
\cite[Lemma~3.3 and Theorem~3.4]{Chen2026Boundaries}, in the forms
needed below.

\begin{lemma}[Endpoint self-loop]
\label{lem:endpoint-loop}
Let \(B\subset\Z_{>0}\) be finite, put \(L=\sigmaf(B)\), and suppose
\[
  \FS(B)=\cS(L,H).
\]
Let \(h=\max H\), with \(h=0\) when \(H=\varnothing\).  If
\[
  x>\max B
  \qquad\text{and}\qquad
  h<x<L-2h,
\]
then
\[
  \FS(B\cup\{x\})=\cS(L+x,H)
\]
and
\[
  |(x+\FS(B))\setminus\FS(B)|=x.
\]
\end{lemma}

For \(q=|H|\), \(E\geq0\), and the same convention for \(h\), put
\begin{equation}
\label{eq:M-def}
  M(h,q,E)=
  \max\left\{
  h,\,
  1+\left\lfloor
  \frac{1+\sqrt{9+8(E+2q+2h)}}{2}
  \right\rfloor
  \right\}.
\end{equation}

\begin{theorem}[Bounded-excess endpoint bridge]
\label{thm:endpoint-bridge}
Let \(B_m=\{b_1<\cdots<b_m\}\subset\Z_{>0}\), put
\(L_m=\sigmaf(B_m)\), and suppose
\[
  \FS(B_m)=\cS(L_m,H),\qquad 2h<L_m.
\]
Let \(m\geq M(h,q,E)\), and let \(x>b_m\).  If
\[
  e_x=|(x+\FS(B_m))\setminus\FS(B_m)|-(m+1)
  \in\{0,\ldots,E\},
\]
then
\[
  x=m+1+e_x
\]
and
\[
  \FS(B_m\cup\{x\})=\cS(L_m+x,H).
\]
Conversely, every \(e\in\{0,\ldots,E\}\) for which
\(m+1+e>b_m\) gives this transition.
\end{theorem}

\section{Bounded increment excess}
\label{sec:bounded-excess}

We begin with a group-theoretic rigidity lemma.

\begin{lemma}
\label{lem:group-stability}
Let \(G\) be an abelian group, let
\[
  R\subseteq G\setminus\{0\},\qquad |R|=m,
\]
and let \(s\geq1\).  Suppose that
\[
  |\FS(R)|\leq m+s,\qquad
  m\geq s+1,\qquad
  m^2>64(m+s-1).
\]
Then there is a finite subgroup \(H\leq G\) such that
\[
  \FS(R)=H,\qquad R\subseteq H,
\]
and
\[
  m+1\leq |H|\leq m+s.
\]
Thus \(R\) is obtained from \(H\setminus\{0\}\) by deleting at most
\(s-1\) elements.
\end{lemma}

\begin{proof}
Put
\[
  H=\operatorname{stab}(\FS(R)),\qquad h=|H|,
  \qquad \rho=|\FS(R)|.
\]
The group \(H\) is finite because it is contained in the finite difference
set \(\FS(R)-\FS(R)\).  If \(h=1\), Theorem~\ref{thm:devos} gives
\[
  \rho\geq1+\frac{m^2}{64}>m+s,
\]
contrary to the hypothesis.  Hence \(h>1\).

Let \(u=|R\cap H|\), and write \(\rho=hq\).  Applying
Theorem~\ref{thm:kneser} to the \(m\) sets \(\{0,r\}\), \(r\in R\),
gives
\[
  \rho\geq h(1+|R\setminus H|)=h(m-u+1).
\]
Since the elements of \(R\cap H\) are nonzero and distinct,
\[
  u\leq h-1.
\]
It follows that
\[
  q\geq m-h+2,\qquad h+q\geq m+2.
\]
If \(q\geq2\), then
\[
  hq\geq2(h+q)-4\geq2m.
\]
On the other hand,
\[
  hq=\rho\leq m+s\leq2m-1,
\]
where the last inequality uses \(m\geq s+1\).  This is a contradiction.
Therefore \(q=1\).  Since \(\FS(R)\) is \(H\)-periodic, contains \(0\),
and has size \(h\), we have \(\FS(R)=H\).  In particular, \(R\subseteq H\).
Finally,
\[
  m+1\leq|\FS(R)|=h\leq m+s,
\]
because \(\{0\}\cup R\subseteq\FS(R)\).
\end{proof}

For \(E\in\Z_{\geq0}\), define
\begin{equation}
\label{eq:I-E}
  I_E=
  \max\left\{
  2E+5,\,
  1+\left\lfloor33+8\sqrt{E+16}\right\rfloor
  \right\}.
\end{equation}
The second term is the least integer strictly larger than the positive
root of
\[
  i^2-66i+65-64E=0.
\]

\begin{theorem}
\label{thm:bounded-excess}
Let \(E\in\Z_{\geq0}\), let \(i\geq I_E\), and let
\[
  A_i=\{a_1<\cdots<a_i\}\subset\R_{>0}.
\]
Then, for every \(e\in\{0,\ldots,E\}\), the equality \(y_i=e\) holds if
and only if there is a unique pair \((t,D)\) with \(t>0\) and
\[
  D\subseteq[i+e-1],\qquad |D|=e,
\]
such that
\begin{equation}
\label{eq:deleted-interval-form}
  A_i=t\bigl([i+e]\setminus D\bigr).
\end{equation}
In this representation \(a_i=t(i+e)\).
\end{theorem}

\begin{proof}
Suppose first that \(y_i=e\leq E\), and put \(x=a_i\).  In the circle
group
\[
  G=\R/x\Z,
\]
consider
\[
  R=\{a_1,\ldots,a_{i-1}\}\pmod{x}.
\]
The set \(R\) has \(i-1\) distinct nonzero elements.  Write
\[
  \rho=|\FS(R)|.
\]
The residues \(0,a_1,\ldots,a_{i-1}\) are distinct, so \(\rho\geq i\).
Each represented residue contains at least one maximal \(x\)-chain of
\(S_{i-1}\).  Lemma~\ref{lem:chain-count} therefore gives
\[
  i\leq\rho\leq d_i=i+e\leq i+E.
\]

Apply Lemma~\ref{lem:group-stability} with
\[
  m=i-1,\qquad s=E+1.
\]
The inequality \(i\geq2E+5\) implies \(m\geq s+1\), while the second term
in the maximum defining \eqref{eq:I-E} gives
\[
  (i-1)^2>64(i+E-1)=64(m+s-1).
\]
Hence \(\FS(R)=H\) for a finite subgroup \(H\leq G\), with
\[
  i\leq h:=|H|\leq i+E,
\]
and \(R\subseteq H\).

The circle group has a unique subgroup of order \(h\), namely
\[
  H=
  \left\{
  0,\frac{x}{h},\frac{2x}{h},\ldots,\frac{(h-1)x}{h}
  \right\}\pmod{x}.
\]
Let \(t=x/h\) and \(\delta=h-i\).  Taking the representatives in
\((0,x)\), we find a set
\[
  D\subseteq[h-1],\qquad |D|=\delta,
\]
such that
\begin{equation}
\label{eq:pre-delta-form}
  A_i=t\bigl([h]\setminus D\bigr).
\end{equation}
We have \(0\leq\delta\leq E\); it remains to show that \(\delta=e\).

Set
\[
  C=[h-1]\setminus D,\qquad |C|=i-1,\qquad L=\sigmaf(C).
\]
Since \(i\geq2E+5\geq2\delta+5\),
\[
  h-1=i+\delta-1\leq\frac{3(i-1)}2-2.
\]
Theorem~\ref{thm:lev} gives
\[
  [2\delta+1,L-(2\delta+1)]_{\Z}\subseteq\FS(C).
\]
By subset-sum symmetry,
\begin{equation}
\label{eq:old-deletion-state}
  \FS(C)=\cS(L,H_D)
  \quad\text{for some}\quad
  H_D\subseteq[1,2\delta]_{\Z}.
\end{equation}

Let \(r=\max H_D\), with \(r=0\) if \(H_D=\varnothing\).  Then \(r<h\).
Moreover,
\[
  L\geq T_{i-1}>h+4\delta\geq h+2r.
\]
The strict inequality \(T_{i-1}>h+4\delta\) follows at the smallest
permitted value \(i=2\delta+5\) from
\[
  T_{i-1}-(h+4\delta)\geq2\delta^2+2\delta+5.
\]
Lemma~\ref{lem:endpoint-loop}, applied with the new element \(h\), now
gives
\[
  |\FS(C\cup\{h\})\setminus\FS(C)|=h.
\]
Scaling by \(t\) and using \eqref{eq:pre-delta-form}, we obtain
\[
  d_i=h=i+\delta.
\]
Thus \(\delta=y_i=e\), which proves \eqref{eq:deleted-interval-form}.

Conversely, suppose \eqref{eq:deleted-interval-form} holds for
\(D\subseteq[i+e-1]\) of size \(e\).  Applying Lev's theorem and
Lemma~\ref{lem:endpoint-loop} as above, with
\[
  h=i+e,\qquad \delta=e,
\]
shows that adjoining \(t(i+e)\) creates \(i+e\) new sums.  Hence
\(y_i=e\).

Finally, \(h=i+e\) and \(t=a_i/h\) are determined by \(A_i\) and \(e\);
the set \(D\) is then recovered from \(t^{-1}A_i\).  This proves
uniqueness.
\end{proof}

\begin{corollary}
\label{cor:bounded-excess-state}
Under the hypotheses of Theorem~\ref{thm:bounded-excess}, put
\[
  h=i+e,\qquad C=[h-1]\setminus D,\qquad L=\sigmaf(C).
\]
Then
\[
  \FS(C)=\cS(L,H_D),\qquad
  \FS(C\cup\{h\})=\cS(L+h,H_D)
\]
for a canonically determined set
\[
  H_D\subseteq[1,2e]_{\Z}.
\]
\end{corollary}

\begin{remark}
\label{rem:bounded-excess-sharp}
The deletion count in Theorem~\ref{thm:bounded-excess} is exact.  For
example,
\[
  [i-1]\cup\{i+e\}
\]
has increment excess \(e\) once \(i\geq2e+5\).  The threshold
\eqref{eq:I-E} is explicit.  When \(E=0\),
Theorem~\ref{thm:cdk-equality} gives the sharper range \(i\geq4\).
\end{remark}

\begin{remark}
\label{rem:small-obstruction}
Some lower bound depending on \(E\) is necessary.  If \(E\geq1\),
\(i=E+2\), and \(M\geq3\), then
\[
  \{1,M,2M,\ldots,(E+1)M\}
\]
has increment excess \(E\), but it is not a dilation of an interval with
\(E\) terms deleted.
\end{remark}

\subsection{Tails of bounded excess}

The local classification also constrains successive prefixes with
bounded excess.

\begin{corollary}
\label{cor:bounded-tail-rigidity}
Fix \(E\in\Z_{\geq0}\), and let
\[
  A=\{a_1<\cdots<a_n\}\subset\R_{>0}.
\]
Suppose that \(i_0\geq I_E\) and
\[
  0\leq y_j\leq E\qquad(i_0\leq j\leq n).
\]
Then there exist a common \(t>0\), a nondecreasing integer sequence
\[
  0\leq e_{i_0}\leq e_{i_0+1}\leq\cdots\leq e_n\leq E,
\]
and sets
\[
  D_j\subseteq[j+e_j-1],\qquad |D_j|=e_j,
\]
such that
\[
  A_j=t\bigl([j+e_j]\setminus D_j\bigr)
  \qquad(i_0\leq j\leq n).
\]
Writing \(h_j=j+e_j\), one has
\begin{equation}
\label{eq:deletion-evolution}
  D_j
  =
  D_{j-1}\cup[h_{j-1}+1,h_j-1]_{\Z}
  \qquad(i_0<j\leq n).
\end{equation}
In particular, the sequence \(e_j\) has at most \(E\) strict increases,
and after its last increase the elements of \(A\) form consecutive
multiples of \(t\).
\end{corollary}

\begin{proof}
Apply Theorem~\ref{thm:bounded-excess} to each prefix.  It gives
\[
  A_j=t_j\bigl([j+e_j]\setminus D_j\bigr),
  \qquad e_j=y_j.
\]
For each \(j\), the coefficient set \([j+e_j]\setminus D_j\) has greatest
common divisor one.  Indeed, if a common divisor were at least two, it would
contain at most half of
\([j+e_j]\), whereas it has \(j\) elements and \(j>e_j\).

Since \(A_{j-1}\subseteq t_j\Z\) and the coefficients of
\(t_{j-1}^{-1}A_{j-1}\) have greatest common divisor one, the quotient
\(t_{j-1}/t_j\) is a positive integer \(q\).  If \(q\geq2\), then the
\(j-1\) elements of \(A_{j-1}\) would correspond to distinct multiples
of \(q\) in \([j+e_j]\).  There are at most
\[
  \left\lfloor\frac{j+e_j}{2}\right\rfloor
  \leq\left\lfloor\frac{j+E}{2}\right\rfloor
  <j-1
\]
such multiples, because \(j\geq i_0+1\geq2E+6\).  This is impossible.
Thus \(q=1\), so \(t_j=t_{j-1}\).  Denote the common scale by \(t\).

The new coefficient \(h_j=j+e_j\) is larger than every coefficient at the
preceding step, so \(h_j>h_{j-1}\).  Hence
\[
  e_j\geq e_{j-1}.
\]
The passage from \(A_{j-1}\) to \(A_j\) adds only \(th_j\).  Every integer
strictly between \(h_{j-1}\) and \(h_j\) must therefore be deleted at
stage \(j\), while the deletion pattern below \(h_{j-1}\) is unchanged.  This
is precisely \eqref{eq:deletion-evolution}.  Finally, a nondecreasing
integer sequence in \([0,E]\) has at most \(E\) strict increases.  Once
\(e_j\) is constant, \(h_j-h_{j-1}=1\), so the subsequent elements are
consecutive multiples of \(t\).
\end{proof}

\begin{remark}
\label{rem:tail-endpoint-state}
The new blocks in \eqref{eq:deletion-evolution} lie above
\(h_{i_0}\), whereas every lower endpoint hole furnished by
Corollary~\ref{cor:bounded-excess-state} lies in \([1,2E]\).  Thus the
deletions below \(h_{i_0}\), and hence the endpoint state, remain fixed
throughout the tail.  The endpoint bridge in
Section~\ref{sec:preliminaries} describes the corresponding persistence
of the endpoint state of the subset-sum supports.
\end{remark}

\section{Effective commensurability and bounded-excess tails}
\label{sec:commensurability}

We first derive effective commensurability from the local classification.

\begin{lemma}
\label{lem:first-rank-jump}
Suppose that \(A=\{a_1<\cdots<a_n\}\subset\R_{>0}\) is
noncommensurable, and let \(r\) be the least index for which \(A_r\) is
noncommensurable.  If \(r\geq4\), then
\[
  y_r=\frac{(r-1)(r-2)}{2}+Y_{r-1}.
\]
Moreover,
\[
  y_i\geq1\qquad(i>r,\ i\geq4).
\]
\end{lemma}

\begin{proof}
All elements of \(S_{r-1}\) lie on one \(\Q\)-line.
The translate \(a_r+S_{r-1}\) lies in a different coset of that line.
The two sets are therefore disjoint, so
\[
  d_r=|S_{r-1}|=1+T_{r-1}+Y_{r-1}.
\]
Subtracting \(r\) from \(d_r\) gives the formula for \(y_r\).

For \(i>r\), the prefix \(A_i\) is noncommensurable because rational rank
cannot decrease.  If in addition \(i\geq4\) and \(y_i=0\), then
Theorem~\ref{thm:cdk-equality} would make \(A_i\) a homogeneous
progression.  Thus \(y_i\geq1\) in the stated range.
\end{proof}

For a fixed integer \(C\), define
\begin{equation}
\label{eq:rank-window-constants}
\begin{split}
  r_C&=
  \max\left(
  \{3\}\cup
  \left\{
  r\in\Z_{\geq4}:
  \frac{r^2-5r+2}{2}\leq C
  \right\}
  \right),\\
  s_C&=\max\{4,r_C\},\\
  D_C^{\mathrm{rk}}&=\max\{0,C+s_C-1\},\\
  J_C^{\mathrm{rk}}&=\max\{I_1,s_C\},\\
  N_C^{\mathrm{rk}}&=J_C^{\mathrm{rk}}+D_C^{\mathrm{rk}}.
\end{split}
\end{equation}
The set in the definition of \(r_C\) is finite, so these constants are
effective.

\begin{theorem}
\label{thm:effective-commensurability}
Let \(C\) be an integer and \(n\geq N_C^{\mathrm{rk}}\).  If
\[
  A\subset\R_{>0},\qquad |A|=n,\qquad
  |\FS(A)|\leq T_n+n+C+1,
\]
then \(A\) is commensurable.
\end{theorem}

\begin{proof}
Suppose otherwise, and let \(r\) be the first rational-rank jump.  If
\(r\geq4\), Lemma~\ref{lem:first-rank-jump} and
\eqref{eq:global-budget} give
\[
\begin{split}
  n+C
  &\geq Y_n\\
  &\geq
  y_r+\sum_{i=r+1}^n y_i\\
  &\geq
  \frac{(r-1)(r-2)}2+n-r.
\end{split}
\]
Hence
\[
  \frac{r^2-5r+2}{2}\leq C,
\]
so \(r\leq r_C\).  This also holds when \(r\leq3\).

Every prefix of size at least \(s_C\) is noncommensurable, and
\(s_C\geq4\).  If \(r\geq s_C\), the formula in
Lemma~\ref{lem:first-rank-jump} gives \(y_r\geq1\); otherwise the second
part of that lemma applies.  Hence \(y_i\geq1\) for \(i\geq s_C\).  The
total excess above one on this tail satisfies
\begin{equation}
\label{eq:rank-surplus}
  \sum_{i=s_C}^n(y_i-1)
  =Y_n-Y_{s_C-1}-(n-s_C+1)
  \leq C+s_C-1
  \leq D_C^{\mathrm{rk}}.
\end{equation}
If \(C+s_C-1<0\), this contradicts the nonnegativity of the left-hand
side.  We may therefore assume that
\(D_C^{\mathrm{rk}}=C+s_C-1\).
Consider the \(D_C^{\mathrm{rk}}+1\) indices
\[
  J_C^{\mathrm{rk}},
  J_C^{\mathrm{rk}}+1,\ldots,
  J_C^{\mathrm{rk}}+D_C^{\mathrm{rk}}.
\]
All these indices are at most \(n\).  If every corresponding excess were at least
two, their contribution to the left-hand side of
\eqref{eq:rank-surplus} would exceed \(D_C^{\mathrm{rk}}\).  Thus
\(y_i=1\) for some \(i\geq I_1\).

Theorem~\ref{thm:bounded-excess}, with \(E=1\), now gives
\[
  A_i=t\bigl([i+1]\setminus\{m\}\bigr)
\]
for some \(t>0\) and \(m\in[i]\).  This prefix is commensurable, contrary
to the definition of \(r\).
\end{proof}

\begin{corollary}
\label{cor:primitive-normalisation}
Under the hypotheses of Theorem~\ref{thm:effective-commensurability},
there is a unique pair \((\lambda,B)\), where \(\lambda>0\) and
\(B\subset\Z_{>0}\) is primitive, such that \(A=\lambda B\).
\end{corollary}

\begin{proof}
Commensurability allows us to clear denominators and divide by a greatest
common divisor.  If
\[
  A=\lambda B=\mu B',
\]
where \(B,B'\) are primitive positive-integer sets, then
\(B'=cB\) for a positive rational number \(c\).  The denominator of \(c\)
must divide \(\gcd(B)=1\), so \(c\) is an integer.  Applying the same
argument to \(B=c^{-1}B'\) shows that \(c^{-1}\) is also an integer.
Therefore \(c=1\).
\end{proof}

For the exceptional case, we also need to control the excesses on a tail.
Let
\[
  B=\{b_1<\cdots<b_n\}\subset\Z_{>0}
\]
be primitive, and retain the notation \(S_i,d_i,y_i,Y_i\).  We call a
prefix \emph{interval-like} if
\[
  S_i=t[0,\ell]_{\Z}
\]
for some \(t>0\) and integer \(\ell\geq0\).

For fixed \(C\), put
\begin{equation}
\label{eq:tail-constants}
\begin{split}
  K_C&=
  \max\left(
  \{2\}\cup
  \left\{
  k\in\Z_{\geq3}:
  \frac{k(k-3)}2\leq C+1
  \right\}
  \right),\\
  D_C&=\max\{0,C+3,C+K_C\},\\
  R_C&=\max\{4,K_C+1\},\\
  E_C&=D_C+1.
\end{split}
\end{equation}

\begin{proposition}
\label{prop:anomaly-tail}
Suppose
\[
  Y_n\leq n+C
  \qquad\text{and}\qquad
  \sigmaf(B)>T_n+n+C.
\]
Then there is an index \(r\leq R_C\) such that
\begin{equation}
\label{eq:anomaly-tail}
  y_i\geq1\quad(i\geq r),
  \qquad
  \sum_{i=r}^n(y_i-1)\leq D_C.
\end{equation}
\end{proposition}

\begin{proof}
The final support is not an arithmetic progression.  Indeed, if
\(S_n=t[0,\ell]_{\Z}\), then \(t\) is a positive integer dividing every
element of the primitive set \(B\), so \(t=1\).  It would follow that
\[
  |\FS(B)|=\sigmaf(B)+1>T_n+n+C+1,
\]
contrary to the support bound.

Suppose first that no interval-like prefix has size at least three.  Then
\eqref{eq:zero-excess} gives \(y_i\geq1\) for \(i\geq4\), and
\[
  \sum_{i=4}^n(y_i-1)
  =Y_n-y_3-(n-3)
  \leq C+3
  \leq D_C.
\]
We may take \(r=4\).

Now suppose that \(k\geq3\) is the largest size of an interval-like
prefix.  Write
\[
  S_k=t[0,\ell]_{\Z}.
\]
The sets \(S_k\) and \(b_{k+1}+S_k\) are disjoint.  Otherwise
\(b_{k+1}/t\) would be an integer not exceeding \(\ell+1\), and their
union would again be an interval of difference \(t\).  Consequently,
\[
  y_{k+1}=|S_k|-(k+1)\geq T_k-k=\frac{k(k-1)}2.
\]
Every later \(y_i\) is at least one by maximality of \(k\) and
\eqref{eq:zero-excess}.  Therefore
\[
  Y_n\geq n-1+\frac{k(k-3)}2.
\]
Together with \(Y_n\leq n+C\), this gives
\[
  \frac{k(k-3)}2\leq C+1,
\]
so \(k\leq K_C\).  Finally,
\[
  \sum_{i=k+1}^n(y_i-1)
  =Y_n-Y_k-(n-k)
  \leq C+k
  \leq D_C.
\]
Taking \(r=k+1\leq R_C\) completes the proof.
\end{proof}

\section{Endpoint entry and tail rigidity}
\label{sec:endpoint-tail}

Fix \(C\geq-1\), and let \(K_C,D_C,R_C,E_C\) be as in
\eqref{eq:tail-constants}.  Choose the least integer \(Q_C\) such that
\begin{equation}
\label{eq:scale-threshold}
  \frac{Q_C(Q_C-1)}2-1>D_C,
\end{equation}
and put
\begin{equation}
\label{eq:entry-constants}
  M_C=M(2,1,E_C),\qquad
  J_C=\max\{I_1,R_C,M_C,Q_C\}.
\end{equation}

\begin{lemma}
\label{lem:one-deletion-support}
Let \(i\geq4\), let \(m\in[i]\), and put
\[
  B=[i+1]\setminus\{m\},\qquad L=\sigmaf(B)=T_{i+1}-m.
\]
Then
\[
  \FS(B)=
  \begin{cases}
    \cS(L,\{1\}),&m=1,\\
    \cS(L,\{2\}),&m=2,\\
    [0,L]_{\Z},&m\geq3.
  \end{cases}
\]
\end{lemma}

\begin{proof}
If \(m\geq3\), the ordered set \(B\) begins with \(1,2\), and every
later element is at most one plus the sum of its predecessors.  The usual
complete-sequence induction therefore gives \(\FS(B)=[0,L]_{\Z}\).

For \(m=1\), the set is \(\{2,3,\ldots,i+1\}\).  Starting with
\(\FS(\{2,3\})=\{0,2,3,5\}\), the same induction shows that every
integer from \(2\) to \(L-2\) is represented.  The sums \(1\) and
\(L-1\) are not represented, and symmetry supplies the stated support.
For \(m=2\), start instead with
\(\FS(\{1,3\})=\{0,1,3,4\}\).  Induction gives every integer from
\(3\) to \(L-3\); the only missing values are \(2\) and \(L-2\).
\end{proof}

\begin{proposition}
\label{prop:endpoint-entry}
Let \(B=\{b_1<\cdots<b_n\}\subset\Z_{>0}\) be primitive and suppose
\[
  |\FS(B)|\leq T_n+n+C+1,\qquad
  \sigmaf(B)>T_n+n+C,
\]
and
\[
  n\geq J_C+D_C.
\]
Then there are \(i_0\in[J_C,J_C+D_C]_{\Z}\) and \(m\in\{1,2\}\) such that
\[
  y_{i_0}=1,\qquad
  B_{i_0}=[i_0+1]\setminus\{m\}.
\]
For every \(i_0\leq i\leq n\), writing \(L_i=\sigmaf(B_i)\), one has
\begin{equation}
\label{eq:persistent-one-hole}
  \FS(B_i)=\cS(L_i,\{m\}).
\end{equation}
\end{proposition}

\begin{proof}
Proposition~\ref{prop:anomaly-tail} supplies \(r\leq R_C\) such that
\eqref{eq:anomaly-tail} holds.  Among the \(D_C+1\) indices
\[
  J_C,J_C+1,\ldots,J_C+D_C
\]
there is an \(i_0\) with \(y_{i_0}=1\).  Otherwise these indices alone
would contribute more than \(D_C\) to the sum in
\eqref{eq:anomaly-tail}.  Since
\(i_0\geq I_1\), Theorem~\ref{thm:bounded-excess} gives
\[
  B_{i_0}=t\bigl([i_0+1]\setminus\{m\}\bigr)
\]
for some \(t>0\) and \(m\in[i_0]\).

The coefficient set \([i_0+1]\setminus\{m\}\) has greatest common
divisor one.  Since its dilation is an integer set, \(t\) is a positive
integer.  Suppose that \(t>1\).
Primitivity of \(B\) gives a first \(j>i_0\) for which \(t\nmid b_j\).
Then
\[
  S_{j-1}\subseteq t\Z,\qquad
  b_j+S_{j-1}\subseteq b_j+t\Z
\]
lie in distinct residue classes modulo \(t\).  Hence
\[
  d_j=|S_{j-1}|=1+T_{j-1}+Y_{j-1}.
\]
It follows that
\[
\begin{split}
  y_j-1
  &=\frac{(j-1)(j-2)}2+Y_{j-1}-1\\
  &\geq\frac{i_0(i_0-1)}2-1\\
  &\geq\frac{J_C(J_C-1)}2-1>D_C,
\end{split}
\]
contrary to \eqref{eq:anomaly-tail}.  Thus \(t=1\).

Lemma~\ref{lem:one-deletion-support} gives an initial state
\[
  H\in\{\varnothing,\{1\},\{2\}\}.
\]
For every \(j>i_0\), Proposition~\ref{prop:anomaly-tail} gives
\[
  1\leq y_j\leq D_C+1=E_C.
\]
The three possible states have
\[
  (h,q)=(0,0),(1,1),(2,1),
\]
and each satisfies \(M(h,q,E_C)\leq M_C\).  At \(i_0\), the required
separation \(2h<L_{i_0}\) holds.  Indeed,
Lemma~\ref{lem:one-deletion-support} gives
\[
  L_{i_0}=T_{i_0+1}-m,
\]
so \(2h\leq4<L_{i_0}\).  Theorem~\ref{thm:endpoint-bridge}
therefore applies successively and preserves \(H\) to the final prefix.

If \(H=\varnothing\), then
\[
  \FS(B)=[0,\sigmaf(B)]_{\Z}.
\]
The support bound would imply
\[
  \sigmaf(B)\leq T_n+n+C,
\]
contrary to the exceptional total condition.  Hence
\(H=\{m\}\) for \(m=1\) or \(m=2\), proving
\eqref{eq:persistent-one-hole}.
\end{proof}

\begin{corollary}
\label{cor:tail-rigidity}
Under the hypotheses of Proposition~\ref{prop:endpoint-entry}, for every
\(j>i_0\),
\begin{equation}
\label{eq:tail-element}
  b_j=j+y_j.
\end{equation}
Moreover,
\begin{equation}
\label{eq:monotone-excess}
  1=y_{i_0}\leq y_{i_0+1}\leq\cdots\leq y_n.
\end{equation}
\end{corollary}

\begin{proof}
For the transition from \(B_{j-1}\) to \(B_j\), the quantity \(e_x\) in
Theorem~\ref{thm:endpoint-bridge} is
\[
  d_j-j=y_j.
\]
The theorem gives \eqref{eq:tail-element}.  At \(i_0\),
\[
  b_{i_0}=i_0+1=i_0+y_{i_0}.
\]
For \(i_0\leq j<n\), the inequality \(b_{j+1}>b_j\) now gives
\[
  j+1+y_{j+1}>j+y_j.
\]
The excesses are integers, so \(y_{j+1}\geq y_j\).
\end{proof}

\section{The partition classification}
\label{sec:classification}

We now prove Theorem~\ref{thm:main-classification}.  For \(C\geq-1\), set
\begin{equation}
\label{eq:global-threshold}
  N_C=
  \max\{1-C,N_C^{\mathrm{rk}},J_C+D_C\}.
\end{equation}
This threshold is effectively computable from the preceding constants.
For \(C=-1\) and \(C=0\), it gives
\[
  N_{-1}=68,\qquad N_0=69.
\]
\begin{proof}[Proof of Theorem~\ref{thm:main-classification}]
Let \(n\geq N_C\) and suppose first that
\[
  A\subset\R_{>0},\qquad |A|=n,\qquad
  |\FS(A)|\leq T_n+n+C+1.
\]
By Theorem~\ref{thm:effective-commensurability}, there is a unique
primitive normalisation
\[
  A=\lambda B,\qquad
  B=\{b_1<\cdots<b_n\}\subset\Z_{>0}.
\]
If
\[
  \sigmaf(B)\leq T_n+n+C,
\]
then the first alternative holds.  We may therefore assume that
\begin{equation}
\label{eq:exceptional-total}
  \sigmaf(B)>T_n+n+C.
\end{equation}

Proposition~\ref{prop:endpoint-entry} supplies \(i_0\) and
\(m\in\{1,2\}\) such that
\[
  B_{i_0}=[i_0+1]\setminus\{m\},
\]
the endpoint state remains \(\{m\}\) through the final prefix, and the
excesses are nondecreasing from \(i_0\).  Let
\[
  k=\max\{i\in[i_0,n]_{\Z}:y_i=1\}
\]
and put \(\ell=n-k\).  If \(\ell=0\), let \(\pi=()\).  Otherwise define
\[
  \pi_r=y_{k+r}-1,\qquad 1\leq r\leq\ell.
\]
By \eqref{eq:monotone-excess},
\[
  1\leq\pi_1\leq\cdots\leq\pi_\ell,
\]
so \(\pi\) is a partition.

For \(i_0<j\leq k\), equations \eqref{eq:tail-element} and \(y_j=1\)
give \(b_j=j+1\).  Hence
\[
  B_k=[k+1]\setminus\{m\}.
\]
For \(1\leq r\leq\ell\), the same identity gives
\[
  b_{k+r}=k+r+1+\pi_r.
\]
Consequently,
\begin{equation}
\label{eq:template-recovered}
  B=\cE_n(m,\pi).
\end{equation}

Let \(w=|\pi|\) and \(L=\sigmaf(B)\).  Direct summation in
\eqref{eq:template-recovered} gives
\begin{equation}
\label{eq:template-total}
  L=T_{n+1}-m+w.
\end{equation}
The final endpoint state is \(\{m\}\), and hence
\begin{equation}
\label{eq:template-cardinality}
  |\FS(B)|=L-1.
\end{equation}
Since \(T_{n+1}=T_n+n+1\), the exceptional total condition
\eqref{eq:exceptional-total} gives
\[
  w>C+m-1,
\]
whereas the support bound and \eqref{eq:template-cardinality} give
\[
  w\leq C+m+1.
\]
As \(w\) is an integer,
\[
  w\in\{C+m,C+m+1\}.
\]
This proves the necessity of the second alternative.

We next prove sufficiency.  The ordinary branch follows from
\eqref{eq:intro-total-bound}.  Fix
\[
  m\in\{1,2\},\qquad
  \pi\in\PiSet(C+m)\cup\PiSet(C+m+1),
\]
and let \(w=|\pi|\), \(\ell=\ell(\pi)\), and \(k=n-\ell\).
Because \(\pi\) is one of the permitted partitions,
\[
  \ell\leq w\leq C+3\leq D_C.
\]
Therefore
\[
  k=n-\ell\geq J_C.
\]

The initial set
\[
  B_k=[k+1]\setminus\{m\}
\]
is primitive and, by Lemma~\ref{lem:one-deletion-support}, satisfies
\[
  \FS(B_k)=\cS(\sigmaf(B_k),\{m\}).
\]
Moreover,
\[
  2m<T_{k+1}-m=\sigmaf(B_k),
\]
because \(k\geq J_C\geq66\).  Each appended element is positive, so
\(2m<\sigmaf(B_j)\) for every subsequent prefix \(B_j\).
For \(1\leq r\leq\ell\), append
\[
  x_r=k+r+1+\pi_r.
\]
Relative to a \((k+r-1)\)-element prefix, this extension has excess
\[
  e_r=1+\pi_r\leq1+w\leq D_C+1=E_C.
\]
The sequence \((x_r)\) is strictly increasing because
\[
  x_1>k+1,\qquad
  x_{r+1}-x_r=1+\pi_{r+1}-\pi_r\geq1.
\]
Since \(k\geq J_C\geq M_C\), the converse part of
Theorem~\ref{thm:endpoint-bridge} applies at each step and preserves the
state \(\{m\}\).  We obtain
\[
  \FS(\cE_n(m,\pi))
  =\cS(T_{n+1}-m+w,\{m\}),
\]
which is \eqref{eq:intro-template-support}.

If \(w=C+m\), equations \eqref{eq:template-total} and
\eqref{eq:template-cardinality} give
\[
  L=T_n+n+C+1,\qquad
  |\FS(B)|=T_n+n+C.
\]
If \(w=C+m+1\), they give
\[
  L=T_n+n+C+2,\qquad
  |\FS(B)|=T_n+n+C+1.
\]
Thus each exceptional template satisfies the target support bound and the
exceptional total condition.

It remains to prove uniqueness.  The primitive normalisation
\((\lambda,B)\) is unique by Corollary~\ref{cor:primitive-normalisation}.
The final lower endpoint hole determines \(m\).  Equivalently,
\(m=1\) precisely when \(1\notin B\), while \(m=2\) precisely when
\(1\in B\) and \(2\notin B\).  In the primitive template, the integer
\[
  k=\max\{j:b_j\leq j+1\}
\]
is the end of the initial block.  The remaining parts are recovered by
\[
  \pi_r=b_{k+r}-(k+r+1).
\]
Hence \(m\) and \(\pi\) are unique.  Finally, the two alternatives are
disjoint by their opposite inequalities for \(\sigmaf(B)\).
\end{proof}

\begin{corollary}
\label{cor:all-fixed-defects}
Theorem~\ref{thm:main-classification}, together with the earlier ranges
in \cite[Theorem~1.1]{CarpenterDefantKravitz2026} and
\cite[Theorems~4.1 and~5.9]{Chen2026Boundaries}, gives an effective
eventual classification for every fixed integer \(C\).
\end{corollary}

\section{Enumeration and finite-state encoding}
\label{sec:consequences}

\begin{proof}[Proof of Corollary~\ref{cor:template-count}]
For \(m=1\), the two allowed partition weights are \(C+1\) and
\(C+2\).  For \(m=2\), they are \(C+2\) and \(C+3\).  The
parametrisation in Theorem~\ref{thm:main-classification} is unique, so the
number of exceptional primitive dilation classes is
\[
  P(C+1)+2P(C+2)+P(C+3).
\]
\end{proof}

\begin{example}
\label{ex:C-minus-one}
For \(C=-1\), the five parameter pairs are
\[
  (1,()),\quad (1,(1)),\quad
  (2,(1)),\quad (2,(1,1)),\quad (2,(2)).
\]
They give the primitive sets
\[
\begin{gathered}
  \{2,3,\ldots,n+1\},\\
  \{2,3,\ldots,n,n+2\},\\
  \{1,3,4,\ldots,n,n+2\},\\
  \{1,3,4,\ldots,n-1,n+1,n+2\},\\
  \{1,3,4,\ldots,n,n+3\}.
\end{gathered}
\]
The first and third have support cardinality \(T_n+n-1\); the other
three have cardinality \(T_n+n\).
\end{example}

\begin{example}
\label{ex:C-zero}
For \(C=0\), the \(m=1\) branch is indexed by the three partitions
\[
  (1),\qquad (2),\qquad (1,1),
\]
and the \(m=2\) branch is indexed by the five partitions whose weights are
\(2\) or \(3\).  Thus there are eight exceptional dilation classes.  This example
also illustrates why the endpoint marker \(m\) must be retained in a
finite encoding: the same partition weight may occur in both branches.
\end{example}

We encode the partition classification by a finite automaton and then
identify its accepted words with the exceptional sets.

For fixed \(C\geq-1\), put
\[
  U_C=C+3,\qquad
  W_{m,C}=\{C+m,C+m+1\}\cap\Z_{\geq0}.
\]
Let the alphabet be
\[
  \mathcal A_C=
  \{\mu_1,\mu_2\}\cup[1,U_C]_{\Z},
\]
where \(\mu_m\) is an initial marker recording the missing endpoint.
Define the finite state set
\[
  \mathcal Q_C=
  \{q_{\mathrm{start}},q_{\mathrm{dead}}\}
  \cup
  \{(m,r,s):
  m\in\{1,2\},\ 0\leq r,s\leq U_C\}.
\]
Let \(\delta\) denote the transition function.  From
\(q_{\mathrm{start}}\), the letter \(\mu_m\) leads to
\((m,0,0)\).  From \((m,r,s)\), an integer letter \(a\) leads to
\((m,a,s+a)\) if
\[
  a\geq\max\{1,r\},\qquad s+a\leq U_C;
\]
every other transition leads to \(q_{\mathrm{dead}}\).  Finally,
\[
  \delta(q_{\mathrm{dead}},a)=q_{\mathrm{dead}}
  \qquad(a\in\mathcal A_C).
\]
A state \((m,r,s)\) is accepting precisely when
\[
  s\in W_{m,C}.
\]
Thus the word \(\mu_m\) alone encodes the empty partition whenever
\(0\in W_{m,C}\).

\begin{theorem}
\label{thm:finite-control}
Fix \(C\geq-1\) and \(n\geq N_C\).  The map
\[
  \mu_m\pi_1\cdots\pi_\ell
  \longmapsto
  \cE_n(m,\pi)
\]
is a bijection from the words accepted by this automaton to the
exceptional primitive sets in Theorem~\ref{thm:main-classification}.
\end{theorem}

\begin{proof}
An accepted word begins with a unique marker \(\mu_m\).  Its remaining
letters form a weakly increasing sequence of positive integers with sum
in
\[
  \{C+m,C+m+1\}.
\]
They therefore form a unique partition
\[
  \pi\in\PiSet(C+m)\cup\PiSet(C+m+1).
\]
The converse direction of Theorem~\ref{thm:main-classification} shows
that \(\cE_n(m,\pi)\) is an exceptional set.  Its forward direction
recovers a unique pair \((m,\pi)\) from every exceptional primitive set,
and hence a unique accepted word.
\end{proof}

\begin{remark}
Although \(n\) changes the absolute positions of the final elements in
\(\cE_n(m,\pi)\), it is not part of the acceptance state.
For fixed \(C\), both the state set and the alphabet are finite.
\end{remark}

\section*{Declarations}

\noindent\textbf{Funding.}
This research received no specific grant.

\medskip
\noindent\textbf{Competing interests.}
The author has no competing interests to declare.

\medskip
\noindent\textbf{Use of artificial intelligence.}
OpenAI Codex assisted with literature searches, language polishing, and manuscript preparation. The author takes full responsibility for the
contents of the article.

\bibliographystyle{amsplain}
\bibliography{references}

@misc{CarpenterDefantKravitz2026,
  author        = {Ruben Carpenter and Colin Defant and Noah Kravitz},
  title         = {Sets with Few Subset Sums},
  year          = {2026},
  eprint        = {2605.05498},
  archivePrefix = {arXiv},
  primaryClass  = {math.CO},
  howpublished  = {arXiv:2605.05498v2},
  note          = {Version of 24 August 2026},
  doi           = {10.48550/arXiv.2605.05498}
}

@misc{Chen2026Boundaries,
  author        = {Lizhong Chen},
  title         = {The First Two Exceptional Boundaries for Sets with Few
                   Subset Sums},
  year          = {2026},
  eprint        = {2609.00044},
  archivePrefix = {arXiv},
  primaryClass  = {math.CO},
  howpublished  = {arXiv:2609.00044v2},
  note          = {Version of 4 September 2026},
  doi           = {10.48550/arXiv.2609.00044}
}

@article{Nathanson1995,
  author  = {Melvyn B. Nathanson},
  title   = {Inverse Theorems for Subset Sums},
  journal = {Trans. Amer. Math. Soc.},
  volume  = {347},
  number  = {4},
  year    = {1995},
  pages   = {1409--1418},
  doi     = {10.1090/S0002-9947-1995-1273512-1}
}

@article{DeVosEtAl2007,
  author  = {Matt DeVos and Luis Goddyn and Bojan Mohar and Robert
             {\v S}{\'a}mal},
  title   = {A Quadratic Lower Bound for Subset Sums},
  journal = {Acta Arith.},
  volume  = {129},
  number  = {2},
  year    = {2007},
  pages   = {187--195},
  doi     = {10.4064/aa129-2-4}
}

@article{Lev1998,
  author  = {Vsevolod F. Lev},
  title   = {On Consecutive Subset Sums},
  journal = {Discrete Math.},
  volume  = {187},
  number  = {1--3},
  year    = {1998},
  pages   = {151--160},
  doi     = {10.1016/S0012-365X(98)80006-X}
}

@article{BrickellSaks1993,
  author  = {Ernest Brickell and Michael Saks},
  title   = {The Number of Distinct Subset Sums of a Finite Set of Vectors},
  journal = {J. Combin. Theory Ser. A},
  volume  = {63},
  number  = {2},
  year    = {1993},
  pages   = {234--256},
  doi     = {10.1016/0097-3165(93)90059-H}
}

@article{BalandraudEtAl2013,
  author  = {{\'E}ric Balandraud and Benjamin Girard and Simon Griffiths and
             Yahya Ould Hamidoune},
  title   = {Subset Sums in Abelian Groups},
  journal = {European J. Combin.},
  volume  = {34},
  number  = {8},
  year    = {2013},
  pages   = {1269--1286},
  doi     = {10.1016/j.ejc.2013.05.009}
}

\end{document}